\documentclass[11pt]{article}
\usepackage{latexsym}
\usepackage{amssymb}
\usepackage{amsmath}
\usepackage{mathrsfs}
\usepackage{enumerate}

\font\Cp = msbm10

\newcommand{\fieldk}{\hbox{{\rm k}}}

\newcommand{\type}{\hbox{\rm deg}}
\newcommand{\LambertW}{\operatorname{W}}

\newcommand{\Rrr}{\hbox{\Cp R}}
\newcommand{\Ccc}{\hbox{\Cp C}}

\newcommand{\qed}{\mbox{$\Box$}\vspace{\baselineskip}}

\newenvironment{proof}{\noindent {\bf Proof:}}{{\qed}}

\newtheorem{theorem}{Theorem}[section]
\newtheorem{proposition}[theorem]{Proposition}
\newtheorem{lemma}[theorem]{Lemma}

\newtheorem{corollary}[theorem]{Corollary}

\newtheorem{conjecture}[theorem]{Conjecture}

\font\Cp = msbm10

\makeatletter
\@addtoreset{equation}{section}
\makeatother

\newcommand{\onethingatopanother}[2]{\genfrac{}{}{0pt}{}{#1}{#2}}

\newcommand{\cd}{\onethingatopanother{c+d=t}{c,d \geq 1}}

\begin{document}

\title{Enumerative and asymptotic analysis of a moduli space}
\author{{\sc Margaret A.\ READDY}}

%%\date{\today}
\date{}

\maketitle

\vspace{-10mm}
\begin{center}
To Doron Zeilberger on the occasion of his 60th birthday
\end{center}

\begin{abstract}
We study exact and asymptotic
enumerative 
aspects of
the Hilbert series of
the cohomology ring of the moduli
space of stable pointed curves of genus zero.
This manifold is related to 
the WDVV (Witten-Dijkgraaf-Verlinde-Verlinde)
equations of string theory.

\vspace{.1in}
\noindent
{\footnotesize
}

%% 05E40    	Combinatorial aspects of commutative algebra
%% 05A16    	Asymptotic enumeration

\end{abstract}

\section{Introduction}

The study of moduli spaces has been a source of much research activity in
physics, topology and algebraic topology.
See for example~\cite{Kontsevich_Manin_I,Lee_Vakil}.
The classical example is the
moduli space of smooth $n$-pointed stable curves 
of genus $g$, denoted $\overline{M}_{g,n}$.
This moduli space
was introduced by 
Deligne, Mumford and 
Knudsen~\cite{Deligne_Mumford, Knudsen, Mumford_stability}
and gives a natural compactification
of Mumford's~\cite{Mumford}  moduli space of nonsingular
curves of genus $g$.
To show how rich this family is,
the case
$\overline{M}_{1,n}$ corresponds to
elliptic curves when $n=1$ 
and to abelian varieties
for $n >1$.

In this paper we focus on combinatorial aspects of the
Hilbert series of the cohomology ring of the moduli space
of stable pointed curves of genus zero.
We show its graded
Hilbert series
satisfies an integral operator identity.
This is used
to give asymptotic behavior, 
and in some cases, exact values, of the coefficients
themselves.
We then study the  total
dimension, that is,
the sum of the coefficients of the Hilbert series.
Its asymptotic behavior surprisingly
involves the Lambert W function,
which has applications to classical tree enumeration,
signal processing and fluid mechanics.

\section{Asymptotics and exact formulas for the cohomology ring 
of the moduli space}

\begin{table}
$$
\begin{array}{c c c c c c  c c c  c c c  c     c     r}
&&   &   &   &    &1   &   &     &   & &&&\hspace{20 mm} &     1 \\
&&   &   &   &1   &    &1  &     &   & &&&               &     2 \\
&&   &   &1  &    &5   &   &1    &   & &&&               &     7 \\
&&   &1  &   &16  &    &16 &     &1  & &&&               &    34 \\
&& 1  &   &42 &    &127 &   &42   &   &1 &&&             &   213 \\
& 1 && 99 && 715 && 715 && 99 && 1 &&                    &  1630 \\
1 && 219 && 3292 && 7723 && 3292 && 219 && 1  &          & 14747 \\
\\
\\
&&&&&&\mbox{(a)}&&&&&&&&\mbox{(b)}       
\end{array}
$$
\caption{The coefficients
of the Hilbert series 
of the cohomology ring of the moduli space $\overline{M}_{0,n}$ 
for $n = 0, \ldots, 6$:
(a) The triangular array of the graded
dimensions
$\alpha_{i,j}$ for $0 \leq i+j \leq 6$
and (b) table for 
the total dimension $\sigma_{n}$ for $0 \leq n \leq 6$.}
\label{table_alphas}
\end{table}

We begin by considering exact enumerative data of the Hilbert series
of the cohomology ring of the moduli space $\overline{M}_{0,n}$.
For completeness, 
we give the physicist's definition
of  this cohomology ring
based on Keel's work.
See~\cite[Section 0.10]{Kontsevich_Manin_II}
and~\cite{Keel}.
For $n \geq 3$
and $\fieldk$ a field of characteristic zero,
denote by $P_n$ the set of {\em stable $2$-partitions
of $\{1, \ldots, n\}$},
that is,
the set of all 
unordered partitions of the elements
$\{1, \ldots, n\}$ into two blocks
$\sigma = S_1 | S_2$ with
$|S_i| \geq 2$.
For each $\sigma \in P_n$,
the element $x_{\sigma}$ corresponds to
a cohomology class of 
$H^*(\overline{M}_{0,n})$.
For $\sigma = S_1| S_2$ and
$\tau =  T_1| T_2$ from the index set~$P_n$,
let
$a(\sigma,\tau)$
 be  the number of nonempty pairwise
distinct sets among $S_i \cap T_j$,
where $1 \leq i, j \leq 2$.
Define the ideal~$I_n$ in the
polynomial ring
$\fieldk[x_{\sigma} :  \sigma \in P_n]$
to be generated by the relations:

\vspace*{.2in}
\begin{enumerate}
\item
(linear relations)
For $i,j,k,l$ distinct:
\begin{equation}
\label{equation_linear_relations}
        R_{ijkl}: \:\:\:\:  \sum_{ij \sigma kl} x_{\sigma}
                    -
                    \sum_{kj \tau il} x_{\tau},
\end{equation}
where the summand
$ij \sigma kl$ means to sum over all stable $2$-compositions
$\sigma = S_1 |S_2$ with
the elements
$i, j \in S_1$ and the elements
$k, l \in S_2$.

\item
(quadratic relations)
For each pair $\sigma$ and $\tau$ with $a(\sigma,\tau) = 4$,
\begin{equation}
\label{equation_quadratic_relations}
        x_{\sigma} \cdot x_{\tau}.
\end{equation}
\end{enumerate}
The 
cohomology ring of the moduli space
of $n$-pointed stable curves of genus $0$
is given by
the quotient ring
$H^*(\overline{M}_{0,n}) = \fieldk[x_{\sigma} \: : \: \sigma \in P_n]/I_n$.

See Table~\ref{table_alphas} for some values.
One immediately sees the coefficients
satisfy Poincar\'e duality.

\begin{proposition}
\label{proposition_symmetric}
The coefficients of the Hilbert series of the moduli space
$\overline{M}_{0,n}$ are symmetric.
\end{proposition}
This follows from the fact that
the genus $g$ moduli space
$\overline{M}_{g,n}$
is a smooth, complete and compact 
variety~\cite{Deligne_Mumford, Knudsen, Mumford_stability}.
For a proof in  the genus $0$ case, see~\cite{Keel}.
By Proposition~\ref{proposition_symmetric}
we thus can
write the Hilbert series as
$$
{\mathcal H}(H^*(\overline{M}_{0,n})) = \sum_{(i,j)} \alpha_{i,j} t^i,
$$
where the sum is over all pairs $(i,j)$
of nonnegative
integers with $i + j = n$.

We next present a recursion for the coefficients of the
Hilbert series of the 
cohomology ring of the moduli
space $\overline{M}_{0,n}$.
The recursion we give here is a symmetrized
version of Keel's~\cite[page 550]{Keel}.
\begin{theorem}
The coefficients of the Hilbert series
series of cohomology ring of the moduli space~$\overline{M}_{0,n}$ 
satisfy the recursion:
\begin{equation}
\label{equation_alphas}
	\alpha_{i+1,j+1} = \alpha_{i+1,j} + \alpha_{i,j+1}
			+
			\frac{1}{2} \sum_{p=0}^i \sum_{q=0}^j
			{i+j + 4 \choose p+q +2} \alpha_{p,q} \cdot
                                                 \alpha_{i-p,j-q},
\end{equation}
where $i, j \geq 0$ and
with the initial conditions $\alpha_{i,0} = \alpha_{0,i} = 1$
for $i \geq 0$.
\end{theorem}

Define the exponential generating function
$$
	f_j(x) = \sum_{i \geq 0} \alpha_{i,j} \frac{x^{i+2}}{(i+2)!}.
$$
These are series formed by taking diagonal entries in the
$\alpha_{i,j}$ triangle
with the power of $t$ slightly shifted.
See Table~\ref{table_f_jays}
for the first few values of
$f_j(x)$.
Let
$I$ be the integral operator
$$
     I(f(x)) = \int_0^x f(t) \:\: dt
$$
and
$I^j$ denote the $j$th iterated integral operator.
The following integral operator identity holds
among these series.

\begin{table}
\label{table_f_jays}
\begin{eqnarray*}
f_{0}(x)
  & = &
       e^{x} - x - 1  \\
f_{1}(x)
  & = &
        2  e^{2x} - \left( \frac{1}{2}  x^2 + 2  x + 2 \right)  e^{x} \\
f_{2}(x)
  & = &
        \frac{27}{2}  e^{3x} 
      - \left(   4x^2 + 20x + 22  \right)  e^{2x}
      + \left( \frac{1}{8}x^4 + \frac{11}{6}x^3 + 
\frac{17}{2}x^2 + 15x + \frac{17}{2}\right)  e^{x} \\
f_{3}(x)
  & = &
        \frac{512}{3}  e^{4x} 
      - \left( \frac{243}{4}x^2 + 324x + 378  \right)  e^{3x}
      + \left( 4x^4 + \frac{160}{3}x^3 + 240x^2 + 432x + 262  \right)  e^{2x} \\
  &   &
      - \left(  \frac{1}{48}x^6 + \frac{2}{3}x^5 + \frac{185}{24}x^4 + 41x^3
               + \frac{423}{4}x^2 + 126x + \frac{164}{3}\right)  e^{x}
\end{eqnarray*}
\caption{Expressions for $f_j(x)$, $j= 0, \ldots, 3$.}
\end{table}

\begin{lemma}
\label{lemma_theodore}
The generating function $f_j(x)$
satisfies the integral operator identity
\begin{equation}
\label{equation_theodore}
	I^{j+1}(f_{j+1}(x)) =
	I^{j+2}(f_{j+1}(x))  + I^{j+1}(f_j(x))
	+
	\frac{1}{2}
	\sum_{q=0}^j I^q(f_q(x)) \cdot I^{j-q}(f_{j-q}(x)),
\end{equation}
where
$f_0$ is given by $f_0 = e^x - x - 1$.
\end{lemma}
\begin{proof}
Multiply the identity~(\ref{equation_alphas})
with $x^{i+j+4}/(i+j+4)!$ and sum over
$i \geq 0$ to give
\begin{eqnarray}
\label{equation_first}
	\sum_{i \geq 0} 
	\frac{\alpha_{i+1,j+1}}{(i+j+4)!}  \cdot x^{i+j+4}
	& = &
	\sum_{i \geq 0}
	\frac{\alpha_{i,j+1}}{(i+j+4)!} \cdot x^{i+j+4}
	+
	\sum_{i \geq 0}
	\frac{\alpha_{i+1,j}}{(i+j+4)!} \cdot x^{i+j+4}
	\\
\label{equation_second}
	& + & 
	\frac{1}{2}
	\sum_{i \geq 0}
	\sum_{p=0}^i \sum_{q=0}^j 
	\frac{\alpha_{p,q}}{(p+q+2)!} \cdot 
	\frac{\alpha_{i-p,j-q}}{(i-p+j-q+2)!} \cdot x^{i+j+4}
\end{eqnarray}
The left-hand side of~(\ref{equation_first})
can be written as
\begin{eqnarray*}
	\sum_{i \geq 0} 
	\frac{\alpha_{i+1,j+1}}{(i+j+4)!}  \cdot x^{i+j+4}
	& = &
	\sum_{i \geq 0}
	\alpha_{i,j+1} \cdot \frac{x^{i+j+3}}{(i+j+3)!}
  	- 
	\alpha_{0,j+1} \cdot \frac{x^{j+3}}{(j+3)!}\\
	& = &
	I^{j+1}(f_{j+1}(x)) - \frac{x^{j+3}}{(j+3)!}.
\end{eqnarray*}
In a similar manner,
the right-hand side of~(\ref{equation_first})
becomes
\begin{eqnarray*}
	\sum_{i \geq 0}
	(\alpha_{i,j+1} + \alpha_{i+1,j})
	\cdot
	\frac{x^{i+j+4}}{(i+j+4)!}
	& = &
	I^{j+2}(f_{j+1}(x))
	+
	I^{j+1}(f_j(x)) - \frac{x^{j+3}}{(j+3)!}.
\end{eqnarray*}
Finally, the triple sum in~(\ref{equation_second})
can be written as
\begin{eqnarray*}
	\sum_{q = 0}^j
	\left(
	\sum_{p \geq 0}
	\alpha_{p,q} \cdot
	\frac{x^{p+q+2}}{(p+q+2)!}
	\right) 
	\cdot 
	\left(
	\sum_{p \geq 0}
	\alpha_{p,j-q} \cdot
	\frac{x^{p+j-q+2}}{(p+j-q+2)!}
	\right) 
	& = &
	\sum_{q = 0}^j
	I^q(f_q(x)) \cdot I^{j-q}(f_{j-q}(x)).	
\end{eqnarray*}	
The result now follows by combining these identities.
\end{proof}

Notice the degrees of the polynomials preceding the
$e^{kx}$ terms in $f_j(x)$ 
depend only on $j$ and~$k$.  This motivates one to define
the {\em degree sequence} 
for functions having the expansion
$f(x) = \sum_{j=0}^k p_j(x) \cdot e^{jx}$
where the
$p_j(x)$ are polynomials in $x$
by 
$
	\deg(f) = (\deg(p_j(x)))_{j=0}^k.
$
Here we define 
the degree of the zero polynomial 
to be
$\deg(0) = - \infty$.

It is straightforward to verify the following proposition.

\begin{proposition}
\label{proposition_operations}
Let $f$ and $g$ have series expansions
$f = \sum_{j=0}^k p_j(x) \cdot e^{jx}$
and
$g = \sum_{j=0}^m q_j(x) \cdot e^{jx}$
with
$\type(f) = (a_0, a_1, \ldots, a_k)$ and
$\type(g) = (b_0, b_1, \ldots, b_m)$.
Then

\vspace*{-2mm}
\begin{enumerate}
\item
[(a)]
$
	\type(f+g) \leq (c_0, \ldots, c_n)
$,
where
$c_i = \max(a_i,b_i)$
and
$n = \max(k,m)$.

\item
[(b)]
$
	\type(f \cdot g) \leq (d_0, \ldots, d_{k+m}),
$
where
${\displaystyle d_i = \max_{p+q = i}(a_p + b_q)}$.

\item
[(c)]
$
	\type(\frac{d f}{dx}) = (a_0 - 1, a_1, \ldots, a_k)
$
with the convention $0 - 1 = - \infty$.

\item
[(d)]
If $f$ is non-zero 
then
$$
	\type(I(f)) = (b_0, a_1, \ldots, a_k)
$$
with 
$b_0 = \max(a_0 + 1, 0)$.

\item
[(e)]
When $f \equiv 0$, that is,
$\type(f) = (-\infty, -\infty, \ldots,  -\infty)$
then
$$
	\type(I(f)) = (0, -\infty, \ldots,  -\infty).
$$

\item
[(f)]
The solution $y$ to the differential equation
$$
	y' = y + f(x)
$$ 
has
degree sequence
$$
	\type(y) = (a_0, \max(a_1+1,0), a_2, \ldots, a_k).
$$

\end{enumerate}
\end{proposition}

The degree sequence properties are now used
to derive an upper bound for the degree sequences
of the series $f_j$.

\begin{proposition}
\label{proposition_ocean}
For $j \geq 1$,
the series
$f_j(x)$ has  the expansion
$$
	f_j(x) = \sum_{k=1}^{j+1} p_{j,k}(x) e^{kx},
$$
where
$p_{j,k}(x)$ is a polynomial of degree at most
$2(j-k+1)$,
that is,
$$
     \type(f_j(x)) \leq (-\infty, 2j, 2(j-1), \ldots, 2, 0).
$$
\end{proposition}
\begin{proof}
It is straightforward to verify
that
$f_0 = e^x -1 -x$
and
$f_1 = 2 e^x - (\frac{1}{2} x^2 + 2x + 2) e^x$.
Hence
$\type(f_0) = (1,0)$
and
$\type(f_1) = (-\infty, 2,0)$.

The proof is by induction on $j$, 
and we have just verified the induction basis $j=1$.
Apply the differential operator
$\frac{d^{j+2}}{dx^{j+2}}$ to
equation~(\ref{equation_theodore}) to obtain
\begin{equation}
\label{equation_hello}
	\frac{d}{dx}(f_{j+1}(x))
	=
	f_{j+1}(x) 
	+
	\frac{d^{j+2}}{dx^{j+2}}
	\left(
		I^{j+1}(f_j(x)) + \frac{1}{2}
		\sum_{q=0}^j I^q(f_q(x)) \cdot
		             I^{j-q}(f_{j-q}(x))
	\right).
\end{equation}
Assume that 
$\type(f_k) \leq (-\infty, 2k, 2k-2, \ldots, 2, 0)$ for all
$1 \leq k \leq j$.
By Proposition~\ref{proposition_operations} we have
\begin{enumerate}
\item
[(i)]
$\type(I^{j+1}(f_j)) \leq (j, 2j, 2j-2, \ldots, 2, 0)$,

\item
[(ii)]
$\type(I^q(f_q) \cdot I^{j-q}(f_{j-q})) \leq 
(j-2, j+ \max(q,j-q) -1, 2j, 2j-2, \ldots, 2, 0)$
for $1 \leq q \leq j-1$,

\item
[(iii)]
$\type(I^j(f_j) \cdot f_0) \leq 
(j, 2j+1, 2j, 2j-2, \ldots,  2, 0)$.

\end{enumerate}
Adding these terms together, 
we obtain the
upper bound
for the
corresponding degree sequence to be
$(j, 2j+1, 2j, 2j-2, \ldots, 2, 0)$.
Applying the operator
$d^{j+2}/dx^{j+2}$, we have the upper bound
$(-\infty, 2j+1, 2j, 2j-2, \ldots, 2, 0)$ for the rightmost term 
in the differential equation~(\ref{equation_hello}).
Hence by Proposition~\ref{proposition_operations} (f),
the degree sequence of $f_{j+1}$ has the desired form.
\end{proof}

In order to prove the next two theorems, we will
need a lemma.

\begin{lemma}
The following two identities hold for $t \geq 1$:
\begin{eqnarray}
\label{equation_c_minus_one}
\sum_{\cd}
             {t \choose c,d}
             c^{c-1} \cdot
             d^{d-1}
  & = &
  2 \cdot (t-1) \cdot t^{t-2} 
\\
\label{equation_c}
\sum_{\cd}
             {t \choose c,d}
             c^{c} \cdot
             d^{d-1}
  & = &
  (t-1) \cdot t^{t-1} 
\end{eqnarray}
\label{lemma_Abel}
\end{lemma}
\begin{proof}
Observe
the left-hand side of equation~(\ref{equation_c_minus_one})
enumerates the number of pairs of rooted labeled trees
on a $t$-element set, that is,
the coefficient of $x^t/t!$ in the generating function
$f(x)^2$, where
$f(x) = \sum_{t \geq 1} t^{t-1} \cdot x^t/t!$.
Noticing that $f(x)$ satisfies the functional
equation
$f(x) = x \cdot e^{f(x)}$, the coefficient can also be determined by
Lagrange inversion 
formula~\cite[Theorem 5.4.2]{Stanley_ECII} as follows:
\begin{equation*}
\left[\frac{x^t}{t!}\right] f(x)^2 
= t! \cdot \frac{2}{t} \left[x^{t-2}\right](e^x)^t 
= 2 \cdot (t-1)! \cdot \frac{t^{t-2}}{(t-2)!} = 2 \cdot (t-1) \cdot t^{t-2}.
\end{equation*}

One can also prove equation~(\ref{equation_c_minus_one})
using the fact the 
Abel polynomials
$p_n(x) = x \cdot (x-na)^{n-1}$
form a {\em sequence of binomial type}, that is,
they
satisfy the relation
$
   p_n(x+y) = \sum_{k=0}^n {n \choose k} p_k(x) p_{n-k}(y).
$
See~\cite{Rota_Kahaner_Odlyzko}.
The desired identity then follows from the substitution
$x = y = 1$, $a = -1$ and $n = t-2$.

To prove~(\ref{equation_c}),
multiply~(\ref{equation_c_minus_one}) by
$c+d = t$ and use that the two
resulting sums are equal.
\end{proof}

We now show the bounds given
in Proposition~\ref{proposition_ocean} are sharp.

\begin{theorem}
\label{theorem_theodore}
For $s \geq 0$, $t \geq 1$,
the coefficient of $x^{2s} e^{tx}$ in
the series
$f_{s+t-1}$ is given by 
\begin{equation}
\label{equation_gamma_s_t}
     [x^{2s} e^{tx}] f_{s+t-1}
         = \frac{(-1)^s}{2^s \cdot s!} 
                       \cdot \frac{t^{2(s+t-1)}}{t!}.
\end{equation}
Furthermore,
the degree sequence of
the series $f_j$ for $j \geq 1$
is equal to 
$$
\deg(f_j(x)) = (- \infty, 2j, 2(j-1), \ldots, 2,0).
$$
\end{theorem}

\begin{proof}
Let
$\gamma_{s,t} = [x^{2s}e^{tx}] f_{s+t-1}$.
Note $\gamma_{s,0} = 0$ for $s \geq 0$.
Observe that for $t \geq 1$
$$
     [x^{2s}e^{tx}] (I^k(f_{s+t-1})) = 
              \frac{\gamma_{s,t}}{t^k}.
$$
Extract the coefficient
of 
$x^{2s} e^{tx}$
from Lemma~\ref{lemma_theodore}
in the case $j + 1 = s+t-1$ 
to give
\begin{eqnarray*}
    \frac{\gamma_{s,t}}{t^{s+t-1}} 
    & = & 
    \frac{\gamma_{s,t}}{t^{s+t}}
    + 0
    +\frac{1}{2} \cdot 2 \cdot [x^{2s} e^{tx}] f_0
                         \cdot I^{s+t-2}(f_{s+t-2})
    \\
    & + &  \frac{1}{2} 
	\sum_{q=1}^{s+t-3} [x^{2s} e^{tx}] I^q(f_q(x)) 
                   \cdot I^{s+t-2-q}(f_{s+t-2-q}(x)).
\end{eqnarray*}
We obtain
\begin{equation}
\label{equation_gamma_recursion}
     \frac{\gamma_{s,t}}{t^{s+t-1}}  \cdot \left(1 - \frac{1}{t}\right) = 
     \frac{1}{2} 
	\sum_{a+b = s} \sum_{\cd}
                 \frac{\gamma_{a,c}}{c^{a+c-1}}
                  \cdot
                 \frac{\gamma_{b,d}}{d^{b+d-1}}
\end{equation}
because
the sum of the
inequalities
$a + c \leq q + 1$ and 
$b+d \leq s+t-2-q+1$
implies
$a+b+c+d \leq s+t$.
Furthermore,
since
$a+b+c+d = s+t$, these two inequalities
are actually equalities, giving the index of summation
as above.

Equation~(\ref{equation_gamma_recursion}) 
can be viewed as a recursion for the
coefficients $\gamma_{s,t}$.  Hence we can prove the
theorem by induction on the quantity $s+t$.
The induction basis is
$s=0$ and $t=1$, which is straightforward to verify.
The induction step is to verify
\begin{eqnarray*}
\frac{(-1)^s}{2^s \cdot s!} \frac{t^{s+t-1}}{t!}
\left(1 - \frac{1}{t}\right)
& = &
\frac{1}{2} \sum_{a+b = s} \sum_{\cd}
             \frac{(-1)^a}{2^a a!} \cdot
             \frac{c^{a+c-1}}{c!} \cdot
             \frac{(-1)^b}{2^b \cdot b!} \cdot
             \frac{d^{b+d-1}}{d!}.
\end{eqnarray*}
Multiply by $2 \cdot (-2)^s \cdot s! \cdot t!$
to obtain
\begin{eqnarray*}
2 (t-1) \cdot t^{s+t-2} 
& = &
             \sum_{a+b = s} \sum_{\cd}
             {s \choose a,b} 
             {t \choose c,d}
             c^{a+c-1} \cdot
             d^{b+d-1} \\
& = &
             \sum_{\cd}
             {t \choose c,d}
             c^{c-1} \cdot
             d^{d-1}
             \sum_{a+b = s} 
             {s \choose a,b} 
             c^{a} \cdot
             d^{b}\\
& = &
             t^s \cdot 
             \sum_{\cd}
             {t \choose c,d}
             c^{c-1} \cdot
             d^{d-1},
\end{eqnarray*}
which is true by
equation~(\ref{equation_c_minus_one})
in Lemma~\ref{lemma_Abel}.

The equality in the degree sequence follows from the inequality
in Proposition~\ref{proposition_ocean} and the fact the coefficients
$\gamma_{s,t}$ are non-zero.
\end{proof}

\begin{corollary}
The asymptotic behavior of $\alpha_{i,j}$
as $i \rightarrow \infty$ is given by
$$
	\alpha_{i,j} \sim \frac{(j+1)^{2j+1}}{j!} \cdot (j+1)^i.
$$       
\end{corollary}

By a similar argument, we can find the next coefficient in the
$f_j(x)$.

\begin{theorem}
\label{theorem_thomas}
For 
$s,t \geq 1$,
the coefficient of $x^{2s-1}  e^{tx}$ in
the series
$f_{s+t-1}$ 
is given by 
\begin{equation}
     [x^{2s-1} e^{tx}] \: f_{s+t-1}(x)
         = (-1)^s \cdot \frac{(5s+9t-8) \cdot t^{2s+2t-4}}
                       {3 \cdot 2^{s-1}(s-1)! (t-1)!}.
\end{equation}
\end{theorem}

\begin{proof}
Using integration by parts we have that
for $t \neq 0$
$$
   I\left( 
             (\gamma \cdot x^{2s} + \delta \cdot x^{2s-1})
           \cdot 
              e^{tx})
    \right)
  =
             \left(\frac{\gamma}{t} \cdot x^{2s} 
               + 
              \left(\frac{\delta}{t} - 2s \cdot \frac{\gamma}{t^{2}}
              \right) \cdot x^{2s-1}\right)
           \cdot 
              e^{tx} 
     +
        \text{lower order terms.}
$$
By iterating operator $I$ $k$ times, we obtain
$$
   I^{k}\left( 
             (\gamma \cdot x^{2s} + \delta \cdot x^{2s-1})
           \cdot 
              e^{tx})
    \right)
  =
             \left(\frac{\gamma}{t^{k}} \cdot x^{2s} 
               + 
              \left(\frac{\delta}{t^{k}} - 2 \cdot k \cdot s 
              \cdot \frac{\gamma}{t^{k+1}}\right)
                 \cdot x^{2s-1}\right)
           \cdot 
              e^{tx} 
     +
        \text{lower order terms.}
$$

Let $\delta_{s,t}$ denote the coefficient
$\delta_{s,t} = [x^{2s-1} e^{t x}] f_{s+t-1}(x)$.
Note that
$\delta_{1,0} = -1$,
$\delta_{s,0} = 0$ for $s \geq 2$
and
$\delta_{0,t} = 0$ for $t \geq 0$.

Consider now the coefficient of
$x^{2s-1} e^{t x}$ in Lemma~\ref{lemma_theodore}
where $j = s+t-2$.
\begin{eqnarray*}
\frac{\delta_{s,t}}{t^{s+t-1}} 
- 2 \cdot (s+t-1) \cdot s \cdot \frac{\gamma_{s,t}}{t^{s+t}}
  & = &
   \frac{\delta_{s,t}}{t^{s+t}}
   - 2 \cdot (s+t) \cdot s \cdot \frac{\gamma_{s,t}}{t^{s+t+1}}
+ 0 \\
& & 
+ [x^{2s-1} e^{t x}]
  \frac{1}{2} \cdot
   \sum_{q=0}^{s+t-2}
        I^{q}(f_{q}(x))
      \cdot
        I^{s+t-2-q}(f_{s+t-2-q}(x)) 
\end{eqnarray*}
We now focus our attention on the last term
in the right-hand side of the above
identity.
For shorthand, we call this term $X$.
\begin{eqnarray*}
X
& = &
     \frac{1}{2} \cdot
     \sum_{q=0}^{s+t-2}
     \sum_{a+b =s} 
      \sum_{c+d=t}
       [x^{2a-1} e^{c x}]  I^{q}(f_{q}(x))
       \cdot
       [x^{2b} e^{d x}]    I^{s+t-2-q}(f_{s+t-2-q}(x)) 
\\
& + & 
     \frac{1}{2} \cdot
     \sum_{q=0}^{s+t-2}
     \sum_{a+b =s} 
      \sum_{c+d=t}
       [x^{2a} e^{c x}]  I^{q}(f_{q}(x))
       \cdot
       [x^{2b-1} e^{d x}]    I^{s+t-2-q}(f_{s+t-2-q}(x)) 
\end{eqnarray*}
In order for
$[x^{2a} e^{c x}]  I^{q}(f_{q}(x))$ to be non-zero,
we need $a+c \leq q+1$.
Similarly,
for 
$[x^{2a-1} e^{c x}]  I^{q}(f_{q}(x))$ to be non-zero,
we need $a+c \leq q+1$.
Thus the  variables of summation must satisfy
$a+c \leq  q+1$ and 
$b+d \leq s+t-q-1$.
Adding these two inequalities gives
an equality.
Hence $q=a+c-1$ and
we can remove the summation over $q$.
Furthermore, the two double summation expressions are equal,
so we obtain
\begin{eqnarray*}
X
&=&
     \sum_{a+b =s} 
      \sum_{c+d=t}
       [x^{2a-1} e^{c x}]  I^{a+c-1}(f_{a+c-1}(x))
       \cdot
       [x^{2b} e^{d x}]    I^{b+d-1}(f_{b+d-1}(x)) 
\\
& = &
     -\frac{\gamma_{s-1,t}}{t^{s+t-2}}
     +\sum_{a+b =s} 
      \sum_{\cd}
       \left(\frac{\delta_{a,c}}{c^{a+c-1}} 
               - 2 (a+c-1) \cdot a \cdot \frac{\gamma_{a,c}}{c^{a+c}}
       \right) 
       \cdot
                    \frac{\gamma_{b,d}}
                         {d^{b+d-1}}
\end{eqnarray*}
Observe the terms corresponding to $d=0$ all vanish.
The only surviving term
corresponding to $c=0$ is when $a=1$, yielding the term
$-\gamma_{s-1,t}/t^{s+t-2}$ above.

To summarize, we have
\begin{eqnarray*}
\frac{\delta_{s,t}}{t^{s+t-1}}
- 2 \cdot (s+t-1) \cdot s \cdot \frac{\gamma_{s,t}}{t^{s+t}}
& = &
\frac{\delta_{s,t}}{t^{s+t}}
- 2 \cdot (s+t) \cdot s \cdot \frac{\gamma_{s,t}}{t^{s+t+1}}
- \frac{\gamma_{s-1,t}}{t^{s+t-2}}
\\
& &
+\sum_{a+b =s} \sum_{\cd}
\left( \frac{\delta_{a,c}}{c^{a+c-1}}
              - 2 (a+c-1) \cdot a \cdot \frac{\gamma_{a,c}}{c^{a+c}}
\right)
\cdot
\frac{\gamma_{b,d}}{d^{b+d-1}}
\end{eqnarray*}
Using Theorem~\ref{theorem_theodore},
we note
\begin{eqnarray*}
- \frac{\gamma_{s-1,t}}{t^{s+t-2}}
& = & 
2s \cdot  \frac{\gamma_{s,t}}{t^{s+t}},
\end{eqnarray*}
thus simplifying the identity to
\begin{eqnarray*}
& & 
\frac{\delta_{s,t}}{t^{s+t-1}}
\cdot
\left(1 - \frac{1}{t}\right)
- 2 \cdot (s+t) \cdot s \cdot \frac{\gamma_{s,t}}{t^{s+t}}
\cdot
\left(1 - \frac{1}{t}\right)
\\
& = &
\sum_{a+b =s} \sum_{\cd}
\left( \frac{\delta_{a,c}}{c^{a+c-1}}
              - 2 (a+c-1) \cdot a \cdot \frac{\gamma_{a,c}}{c^{a+c}}
\right)
\cdot
\frac{\gamma_{b,d}}{d^{b+d-1}}.
\end{eqnarray*}
Observe that this identity is a recursion for
$\delta_{s,t}$.
Hence we prove the theorem by induction on~$s+t$.
The induction basis
$s=t=1$ is straightforward.
The induction step is to verify the following
identity:
\begin{eqnarray*}
& & 
\left(\frac{(-1)^s \cdot (5s+9t-8) \cdot t^{s+t-3}}
     {3 \cdot 2^{s-1} \cdot (s-1)! (t-1)!} 
- 2 \cdot (s+t) \cdot s \cdot \frac{(-1)^s \cdot t^{s+t-2}}{2^s \cdot s! t!}
\right)
\cdot
\left(1 - \frac{1}{t}\right)
\\
& & =
\sum_{a+b =s} \sum_{\cd}
\left( 
      \frac{(-1)^a  \cdot 2 a \cdot (5a + 9c-8) \cdot c^{a+c-2}}
            {3 \cdot 2^{a} \cdot a! c!}
      -  \frac{(-1)^a \cdot 2 a \cdot (a+c-1) \cdot c^{a+c-2}}
                                             {2^a \cdot a! c!}
\right)
\cdot
\frac{(-1)^b d^{b+d-1}}
     {2^b \cdot b! d!}
\\
& & =
\sum_{a+b =s} \sum_{\cd}
\frac{(-1)^a \cdot 2 a \cdot c^{a+c-2}}
     {2^a \cdot a! c!}
\left( \frac{(5a + 9c - 8) }
       {3}
      -  (a+c-1)
\right)
\cdot
\frac{(-1)^b d^{b+d-1}}
     {2^b \cdot b! d!}
\end{eqnarray*}
Multiplying by $3 \cdot (-2)^s  \cdot s! \cdot t!$ 
and simplifying gives
\begin{eqnarray*}
(2s+6t-8) \cdot s \cdot t^{s+t-3} \cdot (t-1)
& = &
\sum_{a+b =s} \sum_{\cd}
{s \choose a,b}
{t \choose c,d}
a \cdot c^{a+c-2} \cdot d^{b+d-1}
\cdot
(2a + 6c - 5) \\
& = &
\sum_{\cd}
{t \choose c,d}
\cdot c^{c-2} \cdot d^{d-1}
\sum_{a+b =s} 
a
\cdot
(2a + 6c - 5)
\cdot
{s \choose a,b}
\cdot
c^a \cdot d^b .
\end{eqnarray*}
By applying the operator
$c \cdot \frac{\partial}{\partial c}$ once,
respectively twice, to the binomial theorem,
we obtain the two identities
\begin{eqnarray}
\label{equation_once}
\sum_{a+b =s} 
a
\cdot
{s \choose a,b}
\cdot
c^a \cdot d^b
  & = &
s \cdot c \cdot (c+d)^{s-1} , \\
\label{equation_twice}
\sum_{a+b =s} 
a^{2}
\cdot
{s \choose a,b}
\cdot
c^a \cdot d^b
  & = &
s \cdot c \cdot (c+d)^{s-1} 
+
s \cdot (s-1) \cdot c^{2} \cdot (c+d)^{s-2} .
\end{eqnarray}
We now have
\begin{eqnarray*}
(2s+6t-8) \cdot s \cdot t^{s+t-3} \cdot (t-1)
& = &
\sum_{\cd}
{t \choose c,d}
\cdot c^{c-2} \cdot d^{d-1}
\left((6c-5) \cdot s \cdot c \cdot (c+d)^{s-1}
\right.
\\
& &
\left.
+ 2\cdot s \cdot c \cdot (c+d)^{s-1}
+ 2\cdot s \cdot (s-1) \cdot c^2 \cdot(c+d)^{s-2}
\right)
\\
&=&
s \cdot t^{s-2}
\cdot
\sum_{\cd}
{t \choose c,d}
\cdot c^{c-1} \cdot d^{d-1}
\left((6c-5) \cdot t + 2t + 2(s-1) \cdot c
\right)
\\
& = &
s \cdot t^{s-2}
\cdot
\left(-3t \cdot 2(t-1) t^{t-2} + (6t+2(s-1))\cdot (t-1)t^{t-1}
\right)
\\
&=& 
s \cdot t^{s-2}
\cdot
(t-1) \cdot t^{t-1}
\cdot
\left(-6 + 6t + 2s - 2
\right),
\end{eqnarray*}
where we have applied
identities~(\ref{equation_once})
and~(\ref{equation_once})
in the first step and Lemma~\ref{lemma_Abel}
in the third step.
This completes the induction.
\end{proof}

\section{Asymptotic behavior of the total graded dimension}

Mathematical physicists often refer to 
the sum of the coefficients of a finite Hilbert series 
as the ``dimension'' of a space.
For example, 
using techniques from
combinatorial commutative algebra~\cite{Readdy_Manin_ring},
it was 
proved that
the dimension of the Losev--Manin ring~\cite{Losev_Manin},
that is, the cohomology
ring arising from the Commutativity Equations of physics,
is given by
$n!$.

Let
$$
	\sigma_n = \sum_{i+j =n}  \alpha_{i,j}
$$
denote the total graded dimension  of the cohomology of the
moduli space
$\overline{M}_{0,n}$.

\begin{proposition}
\label{proposition_sigma_recursion}
The total graded dimension
$\sigma_n$
satisfies the recursion
\begin{equation}
\label{equation_sigma_recursion}
	\sigma_{n+2} = 2 \sigma_{n+1} 
			+ \frac{1}{2}\sum_{i+j=n} {n+4 \choose i+2, j+2}
			\sigma_i \cdot \sigma_j ,
\end{equation}
with the initial condition $\sigma_0 = 1$.
\end{proposition}
\begin{proof}
Sum the recursion~(\ref{equation_alphas})
for $\alpha_{p,q}$ in the case $p+q = n+2$.
\end{proof}

The
Lambert W function
is the multivalued complex function which satisfies
$\LambertW(z) \cdot e^{\LambertW(z)} = z$.
The principal branch $\LambertW_{0}(z)$ and the $-1$ branch 
$\LambertW_{-1}(z)$
are the only two branches
which take on real values.
The Lambert W function
has applications
to classical tree enumeration,
signal processing 
and fluid mechanics.
For more details, see~\cite{Chapeau, Corless}.

For our purposes we will need to  consider
two branches of the Lambert W function which together form
an analytic function near
$-1/e$.

\begin{lemma}
\label{lemma_Puiseux}
Let $\Omega$ be the function defined by
$$
	\Omega(z) = \left\{
	            \begin{array}{ll}                     
                     W_{-1}(z)  &{\rm Im}(z) \geq 0, \\
                    W_{1}(z)    &{\rm Im}(z) < 0.
		    \end{array}
		    \right.
$$
Then $\Omega$ is analytic in 
$\Ccc - (-\infty, -1/e] - [0, +\infty)$.
The Puiseux series of $\Omega$ is given by
$$
	\Omega(z) = \sum_{k \geq 0} \mu_k p^k
             =-1 + p - \frac{1}{3}p^2 + \frac{11}{72} p^3 - \cdots
$$
where $p = -\sqrt{2(ez+1)}$
and the coefficients $\mu_k$ are computed by the recurrences
\begin{eqnarray*}
	\mu_k & = &\frac{k-1}{k+1} \left(\frac{\mu_{k-2}}{2} 
                             + \frac{\alpha_{k-2}}{4}\right)
                - \frac{\alpha_k}{2} - \frac{\mu_{k-1}}{k+1},\\
	\alpha_k & = & \sum_{j =2}^{k-1} \mu_j \cdot \mu_{k+1-j},
\end{eqnarray*}
with
$\mu_0 = -1, \mu_1 = 1, \alpha_0 = 2$
and
$ \alpha_1 = -1$.
\end{lemma}
\begin{proof}
Analyticity follows from piecing together the analytic parts of
the Lambert W function.  The Puiseux series expansion is due
to D.\ Coppersmith.  See~\cite[Section 4]{Corless}.
\end{proof}

Form the exponential generating function for the total graded dimension:
$$
	g(x) = \sum_{n \geq 0} \sigma_n \frac{x^{n+2}}{(n+2)!}.
$$
\begin{proposition}
The function $g(x)$ satisfies the differential equation
\begin{equation}
\label{equation_differential_equation}
	g'(x) = \frac{x + 2 g(x)}{1-g(x)}
\end{equation}
with boundary condition $g(0) = 0$.
\end{proposition}
\begin{proof}
Multiplying~(\ref{equation_sigma_recursion})
by $x^n/(n+4)!$ and summing over $n \geq 0$ gives
\begin{equation}
\label{equation_smiles}
	\sum_{n \geq 0} \frac{\sigma_{n+2} \cdot x^n}{(n+4)!} =
	2 \sum_{n \geq 0} \frac{\sigma_{n+1} \cdot x^n}{(n+4)!}
	+ \frac{1}{2} \left(\frac{1}{x^2} g(x)
                      \right)^2.
\end{equation}
Observe
$g'(x) = \sum_{n \geq 0} \sigma_n \cdot x^{n+1}/(n+1)!$
and
$$
	I(g(x)) = x^4 \left(\sum_{m \geq 0}  
                            \sigma_{m+1} \cdot \frac{x^m}{(m+4)!}
                      \right)
	          + \frac{x^3}{3!}
$$
Thus
we may rewrite~(\ref{equation_smiles})
as
\begin{equation}
\label{equation_clouds}
	g(x) - \frac{x^2}{2} - \frac{2x^3}{3!}
	=
	2 \int_{0}^x g(t) dt
	- \frac{2x^3}{3!}
	+ \frac{1}{2} g(x)^2.
\end{equation}
Differentiating~(\ref{equation_clouds})
gives the desired differential equation.
\end{proof}

This generating function is related to the Lambert
W function.

\begin{theorem}
\label{theorem_generating_Lambert}
The function $g(z)$ is given by
\begin{equation}
\label{equation_gee_whiz}
	g(z) = 1 - (z+2) \left(1 + \frac{1}{\Omega(-e^{-2}(z+2))}
			 \right).
\end{equation}
It is analytic in
$\Ccc - (-\infty, -2] - [e-2, +\infty)$.
\end{theorem}
\begin{proof}
Substitute
$
	g(x) = 1 - (x+2)(u+1)
$
in the differential equation~(\ref{equation_differential_equation}).
We obtain the equation
$$
	\frac{dx}{x+2} = -\left(\frac{1}{u} + \frac{1}{u^2}\right) du.
$$
Integrating gives
$$
	\log|x+2| + C = - \log|u| + \frac{1}{u}.
$$
The boundary condition $g(0) = 0$ 
implies
$u(0) = -1/2$ and $C = -2$.
Hence
$$
	-e^{-2}(x+2) = \frac{1}{u} \cdot e^{\frac{1}{u}}.
$$
As the  Lambert W  function  satisfies
the functional equation 
$\LambertW(z) \cdot e^{\LambertW(z)} = z$,
we have that
$W(-(x+2)/e^2) = 1/u$.
Recall that the value $x=0$ implies $u = -1/2$, so
$W(-2e^{-2}) = -2$.
This implies
$\LambertW_{-1}$ is the correct real branch 
of the Lambert W function to select.
See~\cite{Corless}.
Hence
we obtain the generating function
$g(x) = 1 - (x+2) (1 + 1/\LambertW_{-1}(-e^{-2}(x+2)))$.
The analytic continuation to
$\Ccc -  (-\infty, -2] - [e-2, +\infty)$
follows from Lemma~\ref{lemma_Puiseux}.
\end{proof}

In order to obtain asymptotic 
behavior 
of the total graded dimension, one needs to apply
results of Flajolet and Odlyzko.
See~\cite[Section 4.2]{Flajolet_Odlyzko}
and the overview article~\cite{Odlyzko}.
We follow Flajolet and Odlyzko's notation and
terminology.
Write $f(z) \sim g(z)$ as $z \rightarrow w$ to mean
that $\frac{f(z)}{g(z)} \rightarrow 1$ as $z \rightarrow w$.
For $r, \eta >0$ and $0 < \varphi < \pi/2$
define
$$
    \Delta(r, \varphi, \eta) =
    \{z: \:\: |z| \leq r + \eta, \:\: |\arg(z-r)| \geq \varphi\}.
$$
A function $L(u)$ is of {\em slow variation at $\infty$}
if it satisfies
($i$) there exists a positive real number $u_0$ and
an angle $\varphi$ with $0 < \varphi < \pi/2$ such that
$L(u) \neq 0$ and is analytic for
$\{u : -(\pi - \varphi) \leq \arg (u-u_0) \leq \pi - \varphi\}$
and
($ii$) there exists a function $\epsilon(x)$ defined for
$x \geq 0$
satisfying
$\lim_{x \rightarrow + \infty} \epsilon(x) = 0$ and
for all
$\theta \in [-(\pi - \varphi), \pi - \varphi]$ and
$u \geq u_0$ we have
$$
  \left| \frac{L(ue^{i \theta})}{L(u)} - 1 \right| < \epsilon(u)
   \mbox{   and   }
  \left| \frac{L(u \log^2(u))}{L(u)} - 1 \right| < \epsilon(u).
$$
The following asymptotic theorem appears 
in~\cite[Theorem 5]{Flajolet_Odlyzko}.

\begin{theorem}[Flajolet and Odlyzko]
\label{theorem_Flajolet_Odlyzko}
Assume $f(z)$ is analytic on
$\Delta(r, \varphi, \eta) - \{r\}$ and
$L(u)$ is a function of
slow variation at $\infty$.
If
$\alpha \in \Rrr$
and $\alpha \not\in \{0, 1, 2, \ldots \}$ and
$$
	f(z) \sim (r-z)^{\alpha} \cdot L\left(\frac{1}{r-z}\right)
$$
uniformly as $z \rightarrow r$ for
$z \in \Delta(r, \varphi, \eta) - \{r\}$,
then
$$
	[z^n] f(z) \sim \frac{r^{-n} \cdot n^{-\alpha -1}}
                              {\Gamma(-\alpha)}
			      \cdot
	                       L(n).
$$
\end{theorem}

Applying Theorem~\ref{theorem_Flajolet_Odlyzko}
to the Puiseux series of the analytic function
$\Omega$
gives the following result.

\begin{theorem}
The total dimension of the cohomology ring of the
moduli space
$\overline{M}_{0,n}$  has
asymptotic behavior
$$
        \sigma_{n}
      \sim
        \sqrt{\frac{e}{2\pi}} \cdot
        (e-2)^{-n-2} \cdot n^{-3/2} \cdot (n+2)!  
	\:\:\:\:\mbox{    as $n \rightarrow \infty$}.
$$
\end{theorem}
\begin{proof}
We will apply Theorem~\ref{theorem_Flajolet_Odlyzko}
to the function
$f(z) = g(z) - g(e-2) = g(z) - 1$.
By Lemma~\ref{lemma_Puiseux}
$$
     \Omega(\zeta) = -1 + p + O\left(p^{2}\right),
$$
where
$p = - \sqrt{2(e \zeta + 1)}$.
By inverting this relation we have
\begin{eqnarray*}
     \frac{1}{\Omega(\zeta)}
  & = &
      -1 - p + O\left(p^{2}\right) \\
  & = &
      -1 + \sqrt{2e} \cdot \sqrt{\zeta + \frac{1}{e}}  
      + O\left(\zeta +\frac{1}{e}\right) .
\end{eqnarray*}
In the above,
substitute $\zeta = - e^{-2} (z+2)$,
add $1$ and then multiply with $-(z+2) = -e + (e-2-z)$.
We obtain
$$   f(z) = - e \cdot \sqrt{2e} \cdot \sqrt{\frac{(e-2-z)}{e^{2}}}
            + O\left(e-2-z\right) . $$
In other words,
$$   
     f(z) \sim - \sqrt{2e} \cdot \sqrt{e-2-z}   
$$
uniformly as $z \rightarrow e-2$
for $z$ in a deleted pie-shaped neighborhood
$\Delta(r,\varphi,\eta)$ of
$r = e-2$.
By letting $r = e-2$, $\alpha = 1/2$ and $L$ be the constant function 
$-\sqrt{2e}$, Theorem~\ref{theorem_Flajolet_Odlyzko}
applies. We conclude that
$$
        [x^{n}] g(x)
     =
        [x^{n}] f(x)
     \sim
        -
	\frac{(e-2)^{-n} \cdot n^{-3/2}}{\Gamma(-1/2)} 
        \cdot \sqrt{2e}
     =
       \sqrt{\frac{e}{2\pi}} \cdot
       (e-2)^{-n} \cdot n^{-3/2}    $$
as $n \longrightarrow \infty$.
Since $[x^n] g(x) = \sigma_{n-2}/n!$,
we obtain
$$
        \sigma_{n}
      \sim
        \sqrt{\frac{e}{2\pi}} \cdot
        (e-2)^{-n-2} \cdot (n+2)^{-3/2} \cdot (n+2)!  .
$$
Since $(n+2)^{-3/2} \sim n^{-3/2}$ as
$n \rightarrow \infty$, the result follows.
\end{proof}

An equivalent asymptotic expression appears without
proof in~\cite[Chapter 4, page 194]{Manin}.
See~\cite{Goulden_Litsyn_Shevelev} for 
a recursion for the total dimension of a generalization related
to configuration spaces.

\section{Concluding remarks}

It remains to find
the complete sequence of polynomials
to describe the coefficients in the series~$f_j$.
We make the following conjecture.
 
\begin{conjecture}
For $k$ a non-negative integer and
$s,t \geq k$,
the coefficient of
$x^{2s-k} e^{tx}$ in
$f_{s+t-1}(x)$ is given by
\begin{equation}
     [x^{2s-k} e^{tx}] f_{s+t-1}(x)
         = (-1)^s \frac{t^{2s+2t-2k-2}}{2^{s-k} \cdot (s-k)! (t-k)!}
          \cdot Q_k(s,t),
\end{equation}
where
$Q_k(s,t)$ is a polynomial in the variables 
$s$ and $t$ of degree $k$.
\end{conjecture}
As a special case we have already proved 
$Q_0(s,t) = 1$
and
$Q_1(s,t) = (5s+9t-8)/3$.

In~\cite{Keel} Keel showed the moduli space 
$\overline{M}_{0,n}$
is related to 
$n$-pointed rooted trees of one-dimensional projective spaces.
It is interesting to note that trees also emerge 
in our asymptotic
study of the Hilbert series of the cohomology ring.
See recent work of
Chen, Gibney and Krashen~\cite{Chen_Gibney_Krashen}
for further work on trees of higher-dimensional
projective spaces.

\section{Acknowledgements}

The author would like to thank 
the Institute for Advanced
Study where this paper was completed 
while the author was a Member of the
School of Mathematics during 2010--2011,
as well as
Robert~Corless and David Jeffrey for information
about the Lambert W function.

%%%%%%%%%%%%%%%%%%%%
%
%
% References start here
%
%%%%%%%%%%%%%%%%%%%%

\newcommand{\journal}[6]{{\rm #1,} #2, {\rm #3} {\rm #4} (#5), #6.}
\newcommand{\book}[4]{{\rm #1,} #2, #3, #4.}
\newcommand{\bookf}[5]{{\rm #1,} #2, #3, #4, #5.}
\newcommand{\thesis}[4]{{\rm #1,} ``#2,'' Doctoral dissertation, #3, #4.}
\newcommand{\springer}[4]{{\rm #1,} ``#2,'' Lecture Notes in Math.,
                          Vol.\ #3, Springer-Verlag, Berlin, #4.}
\newcommand{\preprint}[3]{{\rm #1,} #2, preprint #3.}
\newcommand{\progress}[2]{{\rm #1,} #2, work in progress.}
\newcommand{\archive}[3]{{\rm #1,} #2, {\rm #3}.}
\newcommand{\unpublished}[1]{{\rm #1,} unpublished.}
\newcommand{\unpublisheddate}[2]{{\rm #1,} unpublished #2.}
\newcommand{\preparation}[2]{{\rm #1,} #2, in preparation.}
\newcommand{\appear}[3]{{\rm #1,} #2, to appear in {\rm #3}.}
\newcommand{\submitted}[4]{{\rm #1,} #2, submitted to {\rm #3}, #4.}
\newcommand{\AdvancesinMathematics}{Adv.\ Math.}
\newcommand{\DiscreteComputationalGeometry}{Discrete Comput.\ Geom.}
\newcommand{\DiscreteMath}{Discrete Math.}
\newcommand{\EuropeanJournalofCombinatorics}{European J.\ Combin.}
\newcommand{\JCTA}{J.\ Combin.\ Theory Ser.\ A}
\newcommand{\JCTB}{J.\ Combin.\ Theory Ser.\ B}
\newcommand{\JournalofAlgebraicCombinatorics}{J.\ Algebraic Combin.}
\newcommand{\communication}[1]{{\rm #1,} personal communication.}

\newcommand{\collection}[9]{{\rm #1,} #2, 
           in {\rm #3} (#4), #5,
           {\rm #6}, {\rm #7}, #8, #9.}

\vspace{.5in}

{\sc
 \noindent
  Margaret\ A.\ Readdy                          \\
  Department of Mathematics                     \\
  University of Kentucky                        \\
  Lexington, KY 40506-0027                      \\
  {\rm readdy@ms.uky.edu}
}

\end{document}